\documentclass[a4paper,12pt]{article}
\usepackage{amsthm}
\usepackage[all,cmtip]{xy} 
\usepackage{dsfont}
\usepackage{amsmath}
\usepackage{amssymb}

\newlength{\abstractwidth}
\renewcommand{\thanks}[1]{\footnote{#1}} 

\newcommand{\be}{\begin{equation}}
\newcommand{\bea}{\begin{eqnarray}}
\newcommand{\eea}{\end{eqnarray}} \newcommand{\ee}{\end{equation}}
 
 \def\ba{\begin{eqnarray}}
\def\ea{\end{eqnarray}}

\newcommand{\Tr}{\textrm{Tr}}

\newcommand{\del}{\partial}

\newcommand{\ddt}{\frac{\del}{\del t}}
\newcommand{\C}{\mathbb{C}}

\def\C{{\bf C}}

\def\o{\omega}

\def\det{{\rm det}}

\def\log{\,{\rm log}\,}

\def\o{\omega}

\def\al{\alpha}
\def\b{\beta}

\def\o{\omega}

\def\ddt{{\partial\over\partial t}}

\def\ti{\tilde}

\def\R{{\bf R}}
\def\C{{\bf C}}

\def\[{{\bf [}}
\def\]{{\bf ]}}

\def\pl{\partial}

\begin{document}
\newtheorem{theorem}{Theorem}
\newtheorem{proposition}{Proposition}
\newtheorem{lemma}{Lemma}
\newtheorem{corollary}{Corollary}
\newtheorem{definition}{Definition}
\newtheorem{conjecture}{Conjecture}
\newtheorem{example}{Example}
\newtheorem{claim}{Claim}
\newtheorem{rmk}{Remark}

\begin{centering}
 
\textup{\LARGE\bf Remarks on the Yang-Mills flow on a compact K\"ahler manifold}

\vspace{5 mm}

\textnormal{\large Tristan C. Collins and Adam Jacob }

\begin{abstract}
{\small

We study the Yang-Mills flow on a holomorphic vector bundle $E$ over a compact K\"ahler manifold $X$.  We construct a natural barrier function along the flow, and introduce some techniques to study the blow-up of the curvature along the flow.  Making some technical assumptions, we show how our techniques can be used to prove that the curvature of the evolved connection is uniformly bounded away from an analytic subvariety determined by the Harder-Narasimhan-Seshadri filtration of $E$.  We also discuss how our assumptions are related to stability in some simple cases.  
 }

\end{abstract}

\end{centering}
 \begin{normalsize}

\section{Introduction}

A current theme in complex differential geometry is the connection between existence of canonical geometric structures and algebraic stability in the sense of geometric invariant theory (GIT).  This theme is in part motivated by the famous theorem of Donaldson-Uhlenbeck-Yau, which states that the existence of a Hermitian-Einstein connection on an indecomposable holomorphic vector bundle $E$ over a K\"ahler manifold $(X,\omega)$ is equivalent to the stability of $E$ in the sense of Mumford-Takemoto \cite{Don1, UY}.  This theorem was first observed by Narasimhan and Seshadri \cite{NS} in the case of complex curves, by Donaldson for algebraic surfaces \cite{Don1}, and by Uhlenbeck and Yau \cite{UY} in arbitrary dimension.  A heat flow approach to the existence of Hermitian-Einstein connections, related to the Yang-Mills flow, was introduced by Donaldson in \cite{Don1}.  This approach has been extended to several more general settings \cite{BaS, Don3, Simp, Siu}.

It is by now well known that the Yang-Mills flow on an irreducible, holomorphic vector bundle $E$ converges to a Hermitian-Einstein connection if and only if $E$ is stable in the sense of Mumford-Takemoto \cite{Don1, Don3, Siu}.  As a result, it is natural to study the limiting properties of the Yang-Mills flow when $E$ is not stable.  When $\dim_{\C}X =2$, this problem was studied extensively by Daskalopoulos and Wentworth \cite{DW, DW2}.  They found that many limiting properties of the Yang-Mills flow are determined by the algebraic structure of $E$.  More precisely, Daskalopoulos and Wentworth show that away from an analytic bubbling set the Yang-Mills flow converges to a Yang-Mills connection on the direct sum of the stable quotients of the graded Harder-Narasimhan-Seshadri filtration \cite{DW}.  We denote this direct sum of stable quotients by $Gr^{hns}(E)$, and denote the set where $Gr^{hns}(E)$ fails to be free by $Z_{alg}$.  In the later paper \cite{DW2}, Daskalopoulos and Wentworth show that the analytic bubbling set is precisely equal to $Z_{alg}$.  This provides a remarkable and deep connection between the limiting behaviour of the Yang-Mills flow and the GIT of the bundle $E$, see for example \cite{GSZ}.

An interesting open problem is to generalize the results of Daskalopoulos-Wentworth in the case when $\dim_{\C}X >2$.  When $\dim_{\C} X=2$, the singularities of a torsion-free, coherent sheaf occur at isolated points, and hence one can define a natural notion of algebraic multiplicity attached to the set $Z_{alg}$.  The precise result  of \cite{DW2} is that the mass of a bubble at a point $p \in X$ is precisely equal to the algebraic multiplicity at $p$.  The connection between these two quantities is provided by the Riemann-Roch theorem.  By contrast, in dimensions greater than 2, the singularities of a torsion-free coherent sheaf over a K\"ahler manifold $X$ are no longer isolated, and this poses a significant difficulty in generalizing the argument of \cite{DW2}.

For K\"ahler manifolds of arbitrary dimension, the second author partially generalized the work of Daskalopoulos and Wentworth, proving the limiting reflexive sheaf along the Yang-Mills flow is in fact isomorphic to $Gr^{hns}(E)$ \cite{J2, J3}. Precisely, given a subsequence of times $t_j$ along the Yang-Mills flow, following Hong-Tian \cite{HT} we define the analytic singular set (or bubbling set) to be:
\be
Z_{an} =\bigcap_{r>0}\{x\in X\,|\liminf_{j\rightarrow\infty}\,r^{4-2n}\int_{B_r(x)}|F_A(t_j))|^2\o^n\geq \varepsilon\},\nonumber
\ee
A large part of Hong and Tian's paper \cite{HT} is dedicated to proving certain properties of $Z_{an}$, however uniqueness and dependence on the choice of subsequence $t_j$ is left open.  They do show that along such a subsequence the Yang-Mills flow converges smoothly on $X\backslash Z_{an}$, modulo gauge transformations, to a Yang-Mills connection on a limit bundle $E_{\infty}$ on $X\backslash Z_{an}$. In \cite{BaS}, Bando and Siu prove this bundle extends to all of $X$ as a reflexive sheaf $\hat E_\infty$, and the second author proves in \cite{J3} that $\hat{E}_{\infty}\cong Gr^{hns}(E)^{**}$.  Since $E_{\infty}$ is locally free on $X\backslash Z_{an}$, the stalk of $Gr^{hns}(E)$ must be free away from $Z_{an}$. Denote the set where $Gr^{hns}(E)$ fails to be free by $Z_{alg}$; we refer to this set as the algebraic singular set. Then a corollary of the main result of \cite{J3} is that $Z_{alg}\subseteq Z_{an}$.  It is an interesting open problem to prove that $Z_{an} = Z_{alg}$, which, in particular, would imply the uniqueness of $Z_{an}$. In this paper, we provide some partial results towards this goal.

Before stating our main theorem, let us recall some basic definitions. Let $E$ be an indecomposable holomorphic vector bundle over a compact K\"ahler manifold $(X,\omega)$. One can always find a Harder-Narasimhan-Seshadri filtration, 
\be
0=S^0\subset S^1\subset S^2\subset\cdots\subset S^p=E,
\ee
defined to have torsion free, stable quotients $Q^i=S^i/S^{i-1}$. Such a filtration may not be unique, however the direct sum of stable quotients $Gr^{hns}(E):=\bigoplus_i Q^i$ is uniquely determined by the isomorphism class of $E$. It follows that the the {\it algebraic singular set} of $E$,  given explicitly by
\be
Z_{alg}:=\{x\in X|Gr^{hns}(E)_x {\,\,\,\,\rm is\, \,not\, \,free}\,\},\nonumber
\ee
is uniquely determined by the isomorphism class of $E$.  If $E$ is not stable it does not admit a Hermitian-Einstein connection, so we do not expect the Yang-Mills flow to converge smoothly to a limiting Yang-Mills connection.  In particular, we expect that bubbles should form in the limit as $t\rightarrow \infty$.

We now state our main result, which we view as a first step towards proving $Z_{an}=Z_{alg}$. For simplicity, we only state the theorems for Harder-Narasimhan-Seshadri  filtrations of length one, given by $0\subset S\subset E$, with the general case following by induction. Let $\gamma$ be the second fundamental form associated to the subsheaf $S$.
\begin{theorem}
\label{ZanZalg}
Suppose that $E$ is a holomorphic vector bundle with $Gr^{hns}(E) = S^{1} \oplus S^{2}$, where $S^{1} \subset E$ and $S^{2} = E/S^{1}$. 
Let $H(t_j)$ be any sequence of metrics along the Donaldson heat flow. Assume that the following estimates hold uniformly along the flow
\begin{enumerate}
\item[(A)]
$|\nabla^{0} h^{S^{i}}| \leq C\sigma^{-k}{\rm Tr}(h^{S^{i}})$
\item[(B)]
$h^{S^{1}} \geq c\sigma^{k}\|{\rm Tr}(h^{S^{1}})\|_{L^{2}(X)} \quad \text{ and } h^{S^{2}} \leq C \sigma^{-k} \|{\rm Tr}(h^{S^{2}})\|_{L^{2}(X)}$
\item[(C)]
$|\gamma|^{2} \leq C\sigma^{-k}$
\end{enumerate}
for some $k>0$, $i=1,2$.  Here, $\sigma$ is the barrier function constructed in Definition~\ref{def: barrier}. Then the analytic singular set is the same as the algebraic singular set of $E$:
\be
Z_{an} =Z_{alg}.\nonumber
\ee
\end{theorem} 
We make some remarks about our assumptions.  First, the polynomial dependence on $\sigma$ is not important; essentially any estimate in terms of $\sigma$ would suffice.  Secondly, the estimates in condition $(B)$ are clearly stability related.  Indeed, when $E$ is stable, so that $S^{1}=E$ and $S^{1} = \emptyset$, then $(B)$ is a consequence of the work of Uhlenbeck-Yau \cite{UY}. When $S$ and $Q$ are locally free on all of $X$, then assumptions $(B)$ and $(C)$ imply condition $(A)$ (see Proposition~\ref{B-AC}).  We can also show that converse.  Namely, if $S$ and $Q$ are locally free, condition $(A)$ implies both $(B)$ and $(C)$ (see Proposition~\ref{prop: lower bd}), an argument which makes essential use of stability.  We expect that the techniques used in the locally free setting can be generalized to the case when $S, Q$ are not locally free, an issue to which we hope to return in the future.  We point out that the estimate in $(A)$ seems difficult to access, and so estimates $(B)$ and $(C)$ seem to be the essential missing ingredients.  While $(B)$ should follow from stability (we have made some progress showing this), we have been unable to make progress on an a priori estimate for the second fundamental form.  Although we are unable to solve the problem in full generality, we consider our results valuable in the techniques employed in their proofs. In particular, we derive  estimates using a barrier function $\sigma$ constructed from an algebraic object on $E$, namely the determinant bundles of quotients of the Harder-Narasimhan-Seshadri filtration.  Moreover,  we expect the computations in sections~\ref{evolve section} and ~\ref{parabolic} to  be useful in the future. We hope that our techniques will be useful in proving a full generalization of Daskalopoulos and Wentworth's result \cite{DW2}. 

We now outline the main idea in the proof of Theorem \ref{ZanZalg}.  Working away from $Z_{alg}$, we need to bound the evolving curvature along the Yang-Mills flow in $L^p$ for any $p$. Assuming the bound on $|\gamma|^2$, we decompose the curvature onto the stable subsheaf $S$ and stable quotient $Q$, denoting the induced curvature on those sheaves as $F^S$ and $F^Q$. From here we see an $L^p$ bound on the induced curvatures $F^S$ and $F^Q$, along with an $L^p$ bound on the derivative $\nabla\gamma$, will yield Theorem \ref{ZanZalg}. First, we bound $F^S$ and $F^Q$ in $L^{p}_{loc}(X\backslash Z_{alg})$.  In this part of the argument we make use of assumptions $(A), (B)$ and $(C)$, along with a Moser iteration argument using the barrier function $\sigma$. Once these bounds are attained, we compute the heat operator on several important terms, deriving an inequality suited for a parabolic Moser iteration argument, which we use to bound $\nabla\gamma$ in $C^0$ (and as a result $L^p$ for any $p$).

It is in this step we make fundamental use of the stability of $S$ and $Q$. Once these bounds are attained, we compute the heat operator on several important terms, deriving an inequality suited for a parabolic Moser iteration argument, which we use to bound $\nabla\gamma$ in $C^0$ (and as a result $L^p$ for any $p$). The case where the Harder-Narasimhan-Seshadri filtration has general length follows by an induction argument similar to the one used in $\cite{DW}$ and $\cite{J3}$. 

The organization of the paper is as follows.  In Section~\ref{background} we provide background for the basic objects we will need in the proof. In Section~\ref{barrier section} we construct the barrier function $\sigma$ that is used in the analysis in later sections.   The $L^p$ bounds away from $Z_{alg}$ for the induced curvature are derived  in Section \ref{inducedcurvesection}.  It is in this section where we discuss the relationship between the assumptions $(A), (B)$ and $(C)$.  Here we point out the essential use of stability.   Finally, in Section~\ref{evolve section}  we compute the heat operator on several curvature and second fundamental form terms. Those equations are used in Section \ref{parabolic} to apply our parabolic Moser iteration argument, completing the proof of Theorem~\ref{ZanZalg}.

\medskip
\begin{centering}
{\bf Acknowledgements}
\end{centering}
\medskip

Both authors would like to thank D.H. Phong for his advice and support.  The authors would also like to thank V. Tosatti for many helpful comments and stimulating discussions.  The authors thank A. Waldron for some helpful comments, and R. Friedman for several very helpful suggestions.

\section{Background}\label{background}
\subsection{The Yang-Mills flow and the Donaldson heat flow}

We begin with a brief introduction to the Yang-Mills flow, and highlight the importance of its relation to the Donaldson heat flow. 

Let $X$ be a compact K\"ahler manifold, and assume the K\"ahler form $\o$ is normalized so $X$ has volume one.  The Yang-Mills flow is a flow of connections $d_{A}:= d + A$ on $E$, where $d_{A}: E \rightarrow E\otimes \Omega^{1}$.  Because $X$ is a complex manifold, this map decomposes into $(1,0)$ and $(0,1)$ parts.  In particular, the connection coefficients decompose as $A=A'+A''$, where $A'$ represents the $(1,0)$ part and $A''$ represents the $(0,1)$ part of $A$.  Thus $d_{A} = \pl_{A}+ \bar\pl_{A}$, where $\pl_A:=\pl+A'$ and $\bar\pl_A:=\bar\pl+A''$. We say $A$ is {\it integrable} if $\bar\pl_A^2=0$, which implies $\bar\pl_A$ defines a holomorphic structure on $E$. For a fixed metric $H_0$, we say a connection is unitary if it is compatible with the metric,  and we denote the space of integrable unitary connections by $\mathcal{A}^{1,1}$. The curvature of a connection, denoted $F_{A}$, is a section of ${\rm End}(E)\otimes \Omega^{1,1}$, and is defined by:
\be
F_A:=\bar\pl A'+\pl A''+A''\wedge A'+A'\wedge A''.\nonumber
\ee
The Yang-Mills functional $YM:\mathcal{A}^{1,1}\longrightarrow \R$ is defined to be the $L^2$ norm of the curvature:
\be
YM(A):=||F_A||^2_{L^2}.\nonumber
\ee
On a general complex manifold, the Yang-Mills flow is the gradient flow of this functional, and is given by:
\be
\dot A=-d_A^*\,F_A.\nonumber
\ee
On a K\"ahler manifold we can rewrite the equation for the flow using Bianchi's second identity ($d_AF_A=0$) and the K\"ahler identities:
\be
\label{YMF2}
\dot A=i\bar\pl_A\Lambda F_A-i\pl_A\Lambda F_A.
\ee
From this formulation one can check that if $A(0)\in\mathcal{A}^{1,1}$, then $A(t)$ is an integrable, unitary connection for all time $t\in [0,\infty)$. Now if $E$ is stable, it was first shown by Donaldson in \cite{Don1} that the Yang-Mills flow converges to a Hermitian-Einstein connection. However, since we are assuming $E$ is not stable, we do not expect the flow to converge to a limiting Hermitian-Enstien connection. In fact, our main object of study is the set of points on the base manifold $X$ where the curvature blows up along the flow.
 \begin{definition}
  Given a sequence of connections $A(t_j)$ along the Yang-Mills flow, the analytic singular set of $E$ (sometimes called the bubbling set) is defined by:
\be
Z_{an} =\bigcap_{r>0}\{x\in X\,|\liminf_{j\rightarrow\infty}\,r^{4-2n}\int_{B_r(x)}|F_A(t_j)|^2\o^n\geq \varepsilon\}
\ee
for some fixed $0<\varepsilon\leq\varepsilon_0 $, where $\varepsilon_0\ll 1$ depends only on $X$. 
\end{definition}
For a precise definition of $\varepsilon_0$, we direct the reader to the proof of Proposition 6 in \cite{HT}.  In the K\"ahler setting, the Yang-Mills flow is closely related to the Donaldson heat flow.  In fact, it is through this relationship that many important properties of the Yang-Mills flow were first realized, such as long time existence and convergence. 
We follow the viewpoint introduced by Donaldson in \cite{Don1}, and direct the reader to that reference for further detail. Starting with a fixed initial metric $H_0$ on $E$, any other metric $H$ is related to $H_{0}$ by an endomorphism $h=H_0^{-1} H$. Conversely, any positive definite Hermitian endomorphism $h$ defines a metric $H=H_0h$.  
\begin{definition} Let $\mathds{1}$ denote the identity map in End(E). The Donaldson heat flow is a flow of endomorphisms $h=h(t)$ given by:
\be
\label{DHF}
h^{-1}\dot h=-(\Lambda F-\mu(E) \mathds{1}),\nonumber
\ee
with initial condition $h(0)=\mathds{1}$.  Here, $F$ is the curvature of the unitary Chern connection of the metric $H(t)=H_0 h(t)$.
\end{definition} 
A unique smooth solution of the flow exists for all $t\in[0,\infty)$, and on any stable bundle this solution will converge to a smooth Hermitian-Einstein metric \cite{Don1, Don3, Simp, Siu}. In fact, one can use a solution $h(t)$ to \eqref{DHF} to construct a solution $A(t)$ to the Yang-Mills flow. Let $A_0$ be an initial connection in $\mathcal{A}^{1,1}$, and consider the one parameter family of holomorphic structures
$\bar\pl_t=\bar\pl+A_t'',\nonumber$
where $A_t''$ is defined by the action of $w=h^{1/2}$ on $A_0''$. Explicitly, this action is given by:
\be
\label{action}
A_t''=wA_0''w^{-1}-\bar\pl ww^{-1},
\ee
which is equivalent to:
\be
\bar\pl_t:=w\circ\bar\pl_0\circ w^{-1}.\nonumber
\ee
Using this one-parameter family of holomorphic structures and the metric $H_0$, we define a one-parameter family of unitary connections $A_t$, and one can check that $A_t$ evolves by the Yang-Mills flow. Conversely, any one-parameter path in $\mathcal{A}^{1,1}$ along the Yang-Mills flow defines an orbit of the complexified gauge group, which gives rise to a solution of the Donaldson heat flow.  The curvature of $F$ along the Donaldson heat flow is related to the curvature $F_A$ along the Yang-Mills flow by the following relation:
\be
\label{GRL}
F_A=w \,F\,w^{-1}.
\ee
An important consequence of this relationship is that the norm of the curvature along the Yang Mills flow given by the fixed metric $H_0$ is equivalent to the norm of the curvature along the Donaldson heat flow given by the evolving metric $H$. Let $(\cdot)^\dagger$ denote the adjoint of an endomorphism with respect to the fixed metric $H_0$,  and let $(\cdot)^*$ denote the adjoint with respect to the evolving metric $H$. For any endomorphism $M$, these two adjoints are related as follows: $M^\dagger=hM^*h^{-1}$. We then see:
\be
|F|_{H}^2={\rm Tr}(FF^*)={\rm Tr}(w^{-1}F_Aw(w^{-1}F_Aw)^*)={\rm Tr}(F_AhF_A^*h^{-1})={\rm Tr}(F_AF_A^\dagger)=|F_A|^2_{H_0}.\nonumber
\ee
Thus from the point of view of uniform curvature bounds, it suffices to prove bounds along either the Donaldson heat flow or the Yang-Mills flow, provided we always compute the norm with the right metric. 

We conclude this section with a simple curvature bound along the Donaldson heat flow.
\begin{lemma}
Along the Donaldson heat flow, there is a constant $C$ so that $|\Lambda F|_{H(t)}$ is uniformly bounded.
\end{lemma}
\begin{proof}
We have the following simple computation for the heat operator on $|\Lambda F|^2_{H}$ (for details see \cite {J1}):
\be
(\pl_t-\Delta)|\Lambda F|^2_{H}=-|\nabla\Lambda F|_H^2-|\overline{\nabla}\Lambda F|_H^2\leq 0.\nonumber
\ee
The lemma follows from the maximum principle. 
\end{proof}

\subsection{Quotients, filtrations and stability}
In this section we introduce the algebraic singular set, show it is uniquely determined by the isomorphism class of $E$, and provide a local analytic description. We begin by recalling the definitions of slope and stability.

Given a torsion free sheaf $\mathcal{E}$, we can define its first Chern class by $c_1(\mathcal{E}):=c_1(det(\mathcal{E}))$, since $det(\mathcal{E})$ is always a line bundle. The slope of $\mathcal{E}$ is then given by:
\be
\mu(\mathcal{E}):=\frac{1}{rk(\mathcal{E})}\int_Xc_1(det(\mathcal{E}))\wedge\o^{n-1}.\nonumber
\ee
We say $\mathcal{E}$ is stable if for every torsion free subsheaf $\mathcal{F}\subset\mathcal{E}$ the inequality $\mu(\mathcal{F})<\mu(\mathcal{E})$ holds. $\mathcal{E}$ is semi-stable if the weak inequality $\mu(\mathcal{F})\leq\mu(\mathcal{E})$ holds.

Next we introduce the Harder-Narasimhan filtration, and recall some of its properties. The following proposition can be found in \cite{Kob}.
\begin{proposition}[\cite{Kob}, Theorem (7.15)] 
Any torsion-free sheaf $E$ carries a unique filtration of subsheaves
\begin{equation}\label{filter}
0=S^{0} \subset S^{1}\subset S^{2} \subset \cdots \subset S^{p} = E,
\end{equation}
 called the \emph{Harder-Narasimhan} filtration of E, such that the quotients $Q^{i} = S^{i}/S^{i-1}$ are torsion-free and semi-stable.  Moreover, the quotients are slope decreasing, satisfying $\mu(Q^{i}) > \mu(Q^{i+1})$, and the associated graded object $Gr^{hn}(E) := \bigoplus_{i=1}^p Q^{i}$ is uniquely determined by the isomorphism class of $E$.
\end{proposition}
We sometimes abbreviate this filtration as the HN filtration.  For our purposes having semi-stable quotients is not good enough, and we must take the filtration one step further:
\begin{proposition}[\cite{Kob}, Theorem (7.18)] Given a semi-stable sheaf $\mathcal{Q}$, there exists a filtration by subsheaves, called the \emph{Seshadri filtration}:
\begin{equation*}
0 = \tilde{S}^{0} \subset \tilde{S}^{1}\subset \cdots \subset \tilde{S}^{q}=\mathcal{Q},
\end{equation*}
such that $\mu(\tilde{S}^{i}) = \mu(\mathcal{Q})$ for all $i$, and each quotient $\tilde{Q}^{i} = \tilde{S}^{i}/\tilde{S}^{i-1}$ is torsion-free and stable. Furthermore, the direct sum of the stable quotients, denoted $Gr^{s}(\mathcal{Q}) := \bigoplus_{i=1}^q \tilde{Q}^{i}$,  is canonical and uniquely determined by the isomorphism class of $\mathcal{Q}$
\end{proposition}
Combining these two propositions, we can construct the Harder-Narasimhan-Seshadri filtration, by finding a Seshadri filtration for each semi-stable quotient in the HN filtration. We sometimes refer to this double filtration as the HNS filtration. Consider the direct sum of stable quotients:
\begin{equation*}
Gr^{hns}(E) := \bigoplus_{k}\bigoplus_{i} \tilde{Q}^{i}_{k}.
\end{equation*} 
It is not hard to check that the Harder-Narasimhan-Seshadri filtration can be written as a single filtration of $E$ by torsion-free coherent sheaves:
\begin{equation}\label{filter}
0=S^{0}\subset S^{1} \subset \cdots \subset S^{p}=E,
\end{equation}
and in this case, setting $Q^{i} = S^{i}/S^{i-1}$ we have $Gr^{hns}(E) = \bigoplus_{i=1}^{p-1}Q^{i}$.  We are now ready for the following definition:
\begin{definition}
The algebraic singular set is defined to be:
\begin{equation*}
Z_{alg} := \{ x\in X \big| Gr^{hns}(E)_{x} \rm{\, \,is\,\, not\,\, free  }\}.
\end{equation*} 
\end{definition}

We would like to elucidate this definition by providing a useful local description of the algebraic singular set.  For the moment, let us focus on the simple case when $E$ has a stable subsheaf $S$ with stable quotient $Q$, so the HNS filtration of $E$ is given by:
\begin{equation*}
0 \subset S \subset E.
\end{equation*}
In this case, we have the exact sequence of torsion free, coherent sheaves
\begin{equation}\label{seq}
\xymatrix{0 \ar[r]& S \ar[r]^{B}& E \ar[r]^{p}& Q \ar[r]& 0},
\end{equation}
and $Gr^{hns}(E) = S\oplus Q$.  Since $S$ is coherent, over an open set $U$ where $S$ is locally free,  the inclusion $S\hookrightarrow E$ is given by a matrix of holomorphic functions $B = B^{\alpha}{}_{\beta}$.  Moreover, over a sufficiently small open set $U \subset X$,  $S$ has a finite length resolution 
\begin{equation}\label{res1}
\xymatrix{
0\ar[r]& \mathcal{O}_{U}^{\oplus r_{\ell}}\ar[r]& \mathcal{O}_{U}^{\oplus r_{\ell-1}}\ar[r]&\cdots \ar[r]&\mathcal{O}_{U}^{\oplus r_{1}}\ar[r]^{T}& S \ar[r]& 0}.
\end{equation}
Again, where $S$ is locally free, the surjection $\mathcal{O}_{U}^{\oplus r_1}\twoheadrightarrow S$ is given by a matrix of holomorphic functions $T = T^{\gamma}{}_{\delta}$.  The resolution of $S$ gives rise to a resolution for $Q$
\begin{equation*}
\xymatrix{
0\ar[r]& \mathcal{O}_{U}^{\oplus r_{\ell}}\ar[r]& \mathcal{O}_{U}^{\oplus r_{\ell-1}}\ar[r]&\cdots \ar[r]&\mathcal{O}_{U}^{\oplus r_{1}}\ar[r]^{B\circ T}& E \ar[r]& Q \ar[r] &0},
\end{equation*}
where now the map $\mathcal{O}_{U}^{\oplus r_1}\rightarrow E$ is the composition $B\circ T$.  Since this is a map between locally free sheaves, it is determined locally by a matrix of holomorphic functions.  The main technical result we need is the following theorem:
\begin{theorem}[\cite{Kob} Chapter 5, Theorem 5.8]
Let $\zeta$ be a coherent sheaf, $U \subset X$ an open set over which $\zeta$ has a finite resolution
\begin{equation*}
\xymatrix{
0\ar[r]& \mathcal{O}_{U}^{\oplus r_{\ell}}\ar[r]&\cdots \ar[r]& \mathcal{O}_{U}^{\oplus r_{2}}\ar[r]^{h}&\mathcal{O}_{U}^{\oplus r_{1}}\ar[r]& \zeta_{U} \ar[r]& 0}.
\end{equation*}
Then we have the following equality of sets;
\begin{equation}\label{kernel}
\{x\in U | \,\zeta_{x} {\rm \,\,is\,\, not\,\, free } \} = \{ x \in U | \text{ rank}(h(x)) < \max_{y\in U} \text{ rank} (h(y))\}.
\end{equation}
\end{theorem}
In particular, it follows immediately that $Z_{alg}$ is an analytic subset of $X$.  In our setting, we note that any point where $Q_{x}$ is free, the stalk $S_{x}$ is free as well.  In particular, we have,
\begin{corollary}
On a sufficiently small neighborhood, $Z_{alg}$ is given by:
\begin{equation*}
Z_{alg}\cap U = \{ x \in U |\, \text{ rank}(B\circ T(x)) < \max_{y\in U} \text{ rank} (B\circ T)(y)\}
\end{equation*}
\end{corollary}
In the general case we obtain a similar description inductively.  Recall the filtration~\eqref{filter}.  For each $i\leq p$ we let $B_{i}$ be inclusion map
\begin{equation}\label{general seq}
\xymatrix{0 \ar[r]& S^{i-1} \ar[r]^{B_{i}}& S^{i} \ar[r]& Q^{i} \ar[r]& 0},
\end{equation}
Denote $Z_{i}$ the set where $Q^{i}$ fails to be locally free.  The set $Z_{p}$  was described in the simple case above.  Fix an open set $U \subset X\backslash Z_{p}$.  On $U$, $S^{p-1}$ is a vector bundle, and hence we get a resolution of $Q^{p-1}$
\begin{equation*}
\xymatrixcolsep{4pc}\xymatrix{
0\ar[r]& \mathcal{O}_{U}^{\oplus r_{\ell}}\ar[r]\cdots \ar[r]&\mathcal{O}_{U}^{\oplus r_{1}}\ar[r]^{B_{p-1}\circ T_{p-1}}& S^{p-1} \ar[r]& Q^{p-1} \ar[r] &0}.
\end{equation*}
We thus obtain a description of the set $\ti Z_{p-1} =  Z_{p-1} \cap X\backslash Z_{p} $.  Then $Z_{p-1} =\ti Z_{p-1} \cup Z_{p}$ is precisely the set where $Q^{p}\oplus Q^{p-1}$ fails to be locally free.  This continues inductively.  As an example, we will indicate how to obtain a description of $Z_{alg}$ in the case where the $HNS$ filtration of $E$ has three quotients, $0\subset S^1\subset S^2\subset E$. In this case, we have $Gr^{hns}(E) = S^1 \oplus S^2/S^1 \oplus E/S^2$.    Using the above argument we obtain an explicit local description of the set
\begin{equation*}
Z_{2} = \{x \in X \big| \,E/S^2_{(x)} \rm{ \,\,is \,\,not\,\, free } \}. 
\end{equation*}
Now, consider the exact sequence:
\begin{equation*}
\xymatrix{
0\ar[r]&{S}^{1}\ar[r]& S^2\ar[r]&S^2/{S}^{1}\ar[r] &0.
}
\end{equation*}
The main difference between this sequence and the sequence~\eqref{seq} is that $S^2$ is not a vector bundle. However, it is a vector bundle over $X\backslash Z_{2}$.  Thus, working over this open manifold we can find an explicit local description of $Z_{1} = \{ x \in X \big| S^2/{S}^{1}_{(x)} \rm{ \,\,is\,\, not\,\, free } \}$. Since we clearly have $Z_{alg} = Z_{1} \cup Z_{2}$, we have succeeded in obtaining a local description of $Z_{alg}$ in this case.  

\subsection{The induced geometry of subsheaves and quotients sheaves} 

In this section we define induced metrics and provide explicit formulas for the induced connections we will need later on. We recall the exact sequence~\eqref{seq} and restrict ourselves to the open manifold $X\backslash Z_{alg}$. Because the sheaves $S$ and $Q$ are locally free here, the metric $H_0$ on $E$ induces a metric $J$ on $S$ and a metric $M$ on $Q$. For sections $\psi,\phi$ of $S$, we define the metric $J$ as follows:
\be
\langle\phi,\psi\rangle_J=\langle B(\phi),B(\psi)\rangle_{H_{0}}.\nonumber
\ee
In order to define $M$ on $Q$, we note that $H_0$ gives a splitting of $\eqref{seq}$:
\be
\label{splitting}
0\longleftarrow{S}\xleftarrow{{\phantom {X}}{\pi}{\phantom {X}}} {E}\xleftarrow{{\phantom {X}}{p^\dagger}{\phantom {X}}} {Q}\longleftarrow0.
\ee
Here $\pi$ is the orthogonal projection from $E$ onto $S$ with respect to the metric $H_0$. For sections $v,w$ of $Q$, we define the metric $M$ by:
\be
\langle v,w\rangle_M=\langle p^{\dagger}(v),p^{\dagger}(w)\rangle_{H_{0}}.\nonumber
\ee
\begin{definition}
{\rm On $X\backslash Z_{an}$ the sheaves $S^{i}$ and $Q^{i}$ are holomorphic vector bundles. We define an} the induced metric {\rm $J_{i}$ on $S^{i}$, and $K_{i}$ on $Q^{i}$ to be one constructed as above.} 
\end{definition}
Note that on $X\backslash Z_{alg}$ it is equivalent to induce the metric $J_{i-1}$ on $S^{i-1}$ by restricting the metric $J_{i}$ induced on $S^{i}$ to the image of $S^{i-1} \subset S^{i}$. 

Once we have sequence $\eqref{splitting}$, the second fundamental form $\gamma\in\Gamma(X,\Lambda^{0,1}\otimes Hom(Q,S))$ is given by:
\be
\gamma= \bar\pl p^\dagger. \nonumber
\ee 
Of course, by composing with the projection $p$, we can write the second fundamental form as a homomorphism from $S^\perp$ to $S$: $\gamma\circ p= \bar\pl p^\dagger\circ p$. As $p$ is holomorphic, and $p^\dagger\circ p=\mathds{1}-\pi$, we see $\gamma\circ p=\bar\pl(\mathds{1}-\pi)=-\bar\pl\pi.$ By the definition of the induced metric $M$,  working with $\gamma$ and $\gamma\circ p$ are equivalent once we take corresponding norms, so we suppress the map $p$ from our notation.

Suppose now that $h(t)$ is the solution of the Donaldson heat flow on the vector bundle $E$, and let $H(t) = H_{0}h(t)$ denote the  metric.  Then $H(t)$ induces metrics $J(t)$ on any coherent, torsion-free subsheaf $S\subset E$ in the manner described above.  If $S$ is locally free on $X\backslash Z_{alg}$, then we can define a smooth section of ${\rm Hom}(S,S)$ on $X\backslash Z_{alg}$ by
\begin{equation*}
h^{S}(t) := J(0)^{-1}J(t). 
\end{equation*}
We will denote the $h^{Q}(t)$ the analogously defined homomorphism induced by $H(t)$ on the quotient sheaf $Q = E/ S$.

 \section{A barrier function}\label{barrier section}

In this section we construct a natural barrier function which is non-negative, and vanishes precisely on $Z_{alg}$.  As in the previous section, we first assume the HNS filtration of $E$ is given by $0 \subset S\subset E$.  The induced metric $J$ on $S$ is a section of the sheaf $S^{*} \otimes \overline{S}^{*}$; that is
\begin{equation*}
J \in \Gamma(X, S^{*} \otimes \overline{S}^{*}),
\end{equation*}
and this section defines a metric on the complement of $Z_{alg}$.  Ideally, we would like to take the function $\sigma$ to be the norm of the determinant of $J$ regarded as a matrix.  However, $S$ need not be a vector bundle, and so the determinant of $J$ as a matrix is not necessarily a globally defined object.  We get around this as follows. Working over $X\backslash Z_{alg}$, the determinant of the matrix $J$ (as given in local coordinates), is a section of the determinant line  bundle $det(S)^{*}\otimes \overline {\det (S)}^*$. Although $\det J \in \Gamma(X\backslash Z_{alg}, det(S)^{*}\otimes \overline {\det (S)}^*))$ is only defined on $X\backslash Z_{alg}$, we show it extends by zero to a smooth, global section. We accomplish this by finding a local expression which makes the extension clear. Recall the exact sequence~(\ref{seq}).  The induced metric $J$ is obtained from the inclusion $B: S \hookrightarrow E$ and the metric $H_{0}$ on $E$.  To begin with, consider the case in which $S$ is locally free.  Then, given a local trivialization of $S$ over $U$, the metric $J$ is given by
\begin{equation*}
J_{\bar{\eta}\beta} = (H_{0})_{\bar{\beta}\alpha}B^{\alpha}{}_{\gamma}\overline{B^{\beta}{}_{\eta}}.
\end{equation*}
From this expression it is clear that $det(J)$ extends smoothly by zero over the set
\begin{equation*}
B_{sing} = \{x \in X | rank(B(x)) < \max rank(B)\}.
\end{equation*}
Since $S$ is locally free, we clearly have $B_{sing} = Z_{alg}$.  

In the case that $S$ is not locally free, we argue as follows.  Observe that we have an inclusion $\bigwedge^r S \rightarrow \bigwedge^r E$, which factors as
\begin{equation*}
\bigwedge^r S \rightarrow \det(S) := (\bigwedge^r S)^{**} \rightarrow (\bigwedge^r E)^{**} = \bigwedge^r E.
\end{equation*}
The key point is that $\det(S)$ is a locally free sheaf of rank one on all of $X$, in particular, a line bundle.  Now, the metric $H$ on $E$ induces a metric $J_r$ on $\det(S)$, which is a smooth, global section of $\det(S)^* \otimes \overline{\det(S)}^*$, clearly extends $\det( J)$, and vanishes precisely where the map $\psi$ given by
\begin{equation*}
\xymatrix{
(\bigwedge^r S)^{**} \ar[r]^{\psi} &\bigwedge^r E}
\end{equation*}
fails to have rank $1$.  We claim that $Z = \{\psi =0\} = Z_{alg}$.  We clearly have $Z \subset Z_{alg}$.  To see the reverse containment observe that, if $\psi(x) \ne0$, then over an open set $U \ni x$ where $\psi \ne 0$, we have the image of $(\bigwedge^r S)^{**}$ in $ \bigwedge^r E$ defines a point in the Grassmanian $Gr(n,r)$, and hence we can pull back the universal $r$-plane bundle over  $Gr(n,r)$ to find a locally free sheaf $\tilde{S}$ on $X\backslash Z$ of rank $r$ which agrees with $S$ over $X\backslash Z_{alg}$.  However, since $S$ is reflexive and $\tilde{S}$ is locally free on $X\backslash Z$, we have that $\text{Hom}(S, \tilde{S})$ and $\text{Hom}(\tilde{S},S)$ are reflexive, and hence normal (see, e.g. \cite{Kob}, Proposition 5.23 ).  Since $Z_{alg}\backslash Z$ has codimension at least 2,  it follows immediately that $\tilde{S} =S$ over $X\backslash Z$. Since $S$ is locally free over $X\backslash Z$, it follows that $\psi|_{X\backslash Z} = \wedge^r B|_{X \backslash Z}$, and hence $B$ has full rank on $X\backslash Z$.  But, again using that $S$ is locally free on $X\backslash Z$, we must have $\wedge^{r}B=0$ on $Z_{alg}\backslash Z$.  In particular, we must have $Z=Z_{alg}$.  We have proved the following:
\begin{proposition}\label{extension prop}
Let $\zeta$ be the smooth section of $\det(S)^{*} \otimes \overline{\det(S)}^{*}$ over $X\backslash Z_{alg}$ defined by $\zeta = \det(J)$.  Then $\zeta$ extends to a smooth, global section $\sigma$ of $\det(S)^{*} \otimes \overline{\det(S)}^{*}$ .  Moreover, the extension is given explicitly by
\begin{equation}\label{sigma}
\sigma(x) = \left\{ \begin{array}{ll}
	\zeta(x)& {\rm if }\quad x\in X\backslash Z_{alg}\\
	0 &{\rm if }\quad x\in Z_{alg}.
\end{array}
\right.
\end{equation}
In particular, we have the equality of sets $\{\sigma =0\} = Z_{alg}$.
\end{proposition}

Of course, this same analysis carries over immediately to the case of the general filtration~\eqref{filter}.  In this case, working on $X\backslash Z_{alg}$ we have
\begin{equation*}
(J_{i-1})_{\bar{q}p} = (J_{i})_{\bar{\beta}\alpha}(B_{i})^{\alpha}{}_{\gamma}(T_{i})^{\gamma}{}_{p}\overline{(B_{i})^{\beta}{}_{\eta}(T_{i})^{\eta}{}_{q}}.
\end{equation*}
where $B_{i}$ is the map in the sequence~\eqref{general seq}, and $T_{i}$ is the first map in the resolution of $S^{i-1}$.  We leave the details to the reader.  We now define our barrier function on $X$. 
\begin{definition}
Fix metrics $\phi_{i}$ on the line bundles $\det(S_{i})^{*}\otimes \overline{\det(S_{i})}^*$.  We then define
\begin{equation}
\sigma = c\prod_{i=1}^{p-1}|\det(J_{i})|_{\phi_{i}},
\end{equation}
where $\det(J_{i})$ denotes the smooth section $\sigma_{i}$ of Proposition~\ref{extension prop}, and $c>0$ is chosen so that $\max_{X} \sigma =1$.
\end{definition}
The reader can easily verify that again in this case we have the equality of sets \\$\{\sigma=0\}=Z_{alg}$.

Along the Donaldson heat flow, $H(t)$ induces metrics on the sheaves in the Harder-Narasimhan-Seshadri filtration, and hence we get barrier functions $\sigma(t)$ in the above way.  In order to avoid confusion, we define
\begin{definition}\label{def: barrier}
Along the Donaldson heat flow, we denote by $\sigma$ the barrier function induced by the metric $H(0)$.
\end{definition}

\section{$L^p$ estimates for induced curvature}
\label{inducedcurvesection}

In the previous section we constructed a barrier function $\sigma$, which vanishes precisely on $Z_{alg}$. In this section we use the assumptions $(A), (B)$ and $(C)$ to get $L^p$ bounds for $|F^S|$ and $|F^Q|$ on compact subsets away from $Z_{alg}$.  We will also elucidate the relationship between the assumptions when $S$ and $Q$ are locally free.

To begin, recall how $\Lambda F$ decomposes on sub bundles and quotient bundles:
\be
\Lambda F|_S=\Lambda F^S+\Lambda\gamma\wedge\gamma^\dagger\nonumber
\ee
and
\be
\Lambda F|_Q=\Lambda F^Q-\Lambda\gamma^\dagger\wedge\gamma.\nonumber
\ee
Along the Donaldson heat flow the quantity $|\Lambda F(t)|_{H_{t}}$ is bounded.  Moreover, $\Lambda\gamma\wedge\gamma^\dagger$ is a positive operator. Thus, we have that
\begin{equation}\label{missing est}
\Lambda F^{S}(t) \leq \Lambda F(t)|_{S} \qquad \text{ and } \quad \Lambda F^{Q}(t) \geq \Lambda F(t)|_{Q}
\end{equation}
The estimates imply estimates for the normalized endomorphisms  $\ti h^S$ and $\ti h^Q$ defined by
\begin{equation*}
\ti h^S = \frac{h^{S}}{\|h^{S}\|_{L^{2}(X)}} \qquad \ti h^S = \frac{h^{Q}}{\|h^{Q}\|_{L^{2}(X)}}
\end{equation*}
We do note specify the metric in the $L^{2}$ norm, since $\|h^{S}\|_{L^{2}(X, H_{t})} = \|h^{S}\|_{L^{2}(X, H_{0})}$.  Notice that $\ti h^S$ and $\ti h^Q$ define the same curvature terms as $h^S$ and $h^Q$, and hence it suffices to prove bounds for the normalized endomorphisms. We state the main result of this section here. 
\begin{proposition}\label{prop: curv est}
Suppose that $(A), (B)$ and $(C)$ hold along the Donaldson heat flow.  Then, for any $K \Subset X\backslash Z_{alg}$, there holds
\begin{equation*}
|F^{S}|_{L^{p}(K,H_{t})}+|F^{Q}|_{L^{p}(K,H_{t})}+|\gamma|_{L^{p}(K,H_{t})} \leq C(K).
\end{equation*}
\end{proposition}

We begin by proving
\begin{proposition}\label{prop: c0}
There exists uniform constants $C, \gamma >0$ such that
\begin{equation*}
{\rm Tr}(h^{S}) \leq C\sigma^{-\gamma}||{\rm Tr}( h^{S})||_{L^{2}(X)},
\end{equation*}
uniformly on $X\times[0,\infty)$, where $\sigma$ is the cut-off function of Section~\ref{barrier section}.
\end{proposition}
\begin{proof}
We use the Moser iteration.  Recall the standard equation
\be
\label{eq: key}
\Delta_{0} h^S_t=g^{j\bar k}\nabla_{\bar k}h^S_t(h^S_t)^{-1}\nabla^0_jh^S_t+h^S_t(\Lambda F^S_0-\Lambda F^S_t),
\ee
where $\Delta_{0} = g^{j\bar{k}}\nabla_{\bar{k}}\nabla^{0}_{j}$.  Note that
\begin{equation*}
|\Lambda F^{S}_{0}| \leq C\sigma^{-k}
\end{equation*}
for $k \gg 0$, which follows easily from the formula for $\Lambda F^{S}_{0}$ in terms of the induced metric $J(0)$. Taking the trace of  equation~\eqref{eq: key}, and using the upper bound for $\Lambda F^{S}(t)$, we have
\begin{equation*}
\Delta \Tr(h^{S}) \geq - (C + |\Lambda F^{S}_{0}|_{H_{0}})\Tr(h^{S}) \geq -C\sigma^{-k}.
\end{equation*}
Let $u = \Tr(h^{S})$. For any $\epsilon >0$, set $\sqrt{\eta_{\epsilon}} = \max\{\sigma^{k} -\epsilon^{k},0\}$ where $k$ is the constant appearing above.  Denote $K(\epsilon) = \{ \sigma \geq \epsilon\}$ .  Then we have
\begin{equation*}
\int_{K(\epsilon)} \Delta u (u^{\alpha} \eta_{\epsilon}^{2}) dVol \geq -C\int_{K(\epsilon)} \sigma^{-k}u^{\alpha+1}\eta_{\epsilon}^{2} dVol
\end{equation*}
Integration by parts proves
\begin{equation*}
\frac{4\alpha}{(\alpha+1)^{2}}\int_{K(\epsilon)} \eta_{\epsilon}^{2} |\nabla u^{(\alpha+1)/2}|^{2}dV  \leq \int_{K(\epsilon)}\left\{C\sigma^{-k}u^{\alpha+1}\eta^{2}_{\epsilon} + \frac{2C}{(\alpha+1)}u^{(\alpha+1)/2}\eta_{\epsilon}|\nabla u^{(\alpha+1)/2}|\right\}dV 
\end{equation*}
By our choice of $k$, for any $\epsilon \geq 0$ we have
\begin{equation*}
\sigma^{-k}\eta^{2}_{\epsilon} \leq C \quad \text{ on } K(\epsilon)
\end{equation*}
for a uniform constant $C$, independent of $\epsilon$.  Thus,
\begin{equation*}
\int_{K(\epsilon)} \eta_{\epsilon}^{2} |\nabla u^{(\alpha+1)/2}|^{2} \leq \frac{C(\alpha+1)^{2}}{\alpha} \int_{K(\epsilon)}u^{\alpha+1}.
\end{equation*}
The Sobolev imbedding theorem, this implies that for any $\delta>0$ there holds
\begin{equation*}
\|u\|^{p}_{L^{p\beta}(K(\epsilon +\delta))} \leq \delta^{-k}C(p)\| u\|_{L^{p}(K(\epsilon))}
\end{equation*}
for $\beta = n/(n-1) >1$.  Fix $\epsilon >0$, and define $\delta_{j} = 2^{-j}\epsilon$.  Then a standard iteration argument proves
\begin{equation*}
\log(\|u \|_{L^{\infty}(K(2\epsilon))}) \leq C_{1}(p)  -\gamma(p) \log(\epsilon) + \log\|u \|_{L^{p}(K(\epsilon))}.
\end{equation*}
In particular, taking $p=2$ we have
\begin{equation*}
|u|_{C^{0}(K(2\epsilon))} \leq \epsilon^{-\gamma}C \|u\|_{L^{2}(X)},
\end{equation*}
for some fixed constants $\gamma, C$.  We claim that this implies that $|u| \leq \sigma^{-\gamma}\|u\|_{L^{2}(X)} C'$.  Suppose that this inequality does not hold.  Then there exists a sequence of times $t_{j}$ and points $x_{j}$ so that
\begin{equation*}
|u|(x_{j}, t_{j}) \geq j\sigma(x_{j})^{-\gamma}
\end{equation*}
Let $2\epsilon_{j} = \sigma(x_{j})$.  Then we have
\begin{equation*}
j\epsilon_{j}^{-\gamma} = j\sigma(x_{j})^{-\gamma} \leq |u|(x_{j}, t_{j}) \leq |u|_{C^{0}(K(2\epsilon_{j}))} \leq C \epsilon_{j}^{-\gamma}
\end{equation*}
and this is a contradiction for $j$ sufficiently large.
\end{proof}

At this point we have proven that $\ti h^S$ is bounded in $C^0$ and by assumption $(A)$, we have a gradient estimate for the rescaled endomorphisms $\ti h^{S}$ on any compact subset away from $Z_{alg}$.  We can now easily deduce the $L^{p}_{2}$ estimates for $\ti h^{S}$.
\begin{lemma}\label{lem: lp2}
Let $K \Subset X\backslash Z_{alg}$ be a compact set.  Then for each $p$ there is a constant $C(p,K)$, independent of time, such that
\begin{equation*}
\| \ti h\|_{L^{p}_{2}(K, H_{0})} \leq C(p,K).
\end{equation*}
\end{lemma}
\begin{proof}
Note that the estimate in Proposition~\ref{prop: c0} implies that
\begin{equation*}
\ti h \leq C(K) \frac{\|{\rm Tr}(h^{S})\|_{L^{2}(X)}}{\|h^{S}\|_{L^{2}(X)}} \leq C'(K).
\end{equation*}
Combining this with assumptions $(A), (B), (C)$, we have that $\Delta_{0}\ti h^{S}$ is uniformly bounded on $K$.  The elliptic theory, combined with the fact that $\| \ti h^{S}\|_{L^{2}(X,H_{0})}$=1 implies uniform $L^{p}_{2}(K, H_{0})$ estimates.
\end{proof}

We now show that, when $S$ and $Q$ are locally free on $X$, then condition $(A)$ is sufficient to deduce conditions $(B)$ and $(C)$

\begin{proposition}\label{prop: lower bd}
Suppose that $S$ and $Q$ are locally free on $X$, and assume that condition $(A)$ holds. Then there exists a uniform constant $C$ such that
\begin{equation*}
{\rm Tr} \left( (\ti h^{S})^{-1} \right) \leq C, \qquad \text{ and } \quad |\gamma|^{2} \leq C
\end{equation*}
\end{proposition}
\begin{proof}

Combining assumption $(A)$ with the the $C^{0}$ estimate of Proposition~\ref{prop: c0} we obtain a uniform $C^1(H_{0})$ bound.  Thus, by the Arzel\`a-Ascoli theorem there exists a subsequence that converges in $C^0$ to a limiting map $h^S_\infty$. 

Recall the result of $\cite{HT}$, which states that along the Yang-Mills flow, one can find a subsequence such that $A_j\longrightarrow A_\infty$ on $X\backslash Z_{an}$ in $C^\infty$. The limiting connection is a Hermitian-Yang-Mills connection on a limiting bundle $E_\infty$ defined on $X\backslash Z_{an}$, and by a result of Bando and Siu from \cite{BaS} this bundle $E_\infty$ extends to a limiting reflexive sheaf defined on all of $X$. By applying the convergence results to the subbundle $S$, it follows that $A^S_j\longrightarrow A^S_\infty$ as well, where $A^S_\infty$ is a limiting connection on the reflexive sheaf $S_\infty$. By the work of the first author $S_\infty$ is in fact isomorphic to $S^{**}=S$ and is stable. As a result $S_\infty$ is a vector bundle. 

Our limiting endomorphism $h^S_\infty$ is a $C^0$ limit of a sequence of functions all normalized to have $L^2$ norm one. Since $X$ is compact, uniform convergence implies convergence in $L^2$, thus $h^S_\infty$ is nonzero.  Define $w^S_\infty$ by $(w^S_\infty)^2=h^S_\infty$, and notice it is nonzero as well. As in $\cite{J3}$, we show that in fact $w^S_\infty$ is a holomorphic map from $S_\infty$ to $S$, and because both are stable bundles of the same slope the map must be an isomorphism. Thus $w^S_\infty$ does not degenerate at any point in $X$, and neither does $h^S_\infty$. 

Recall equation $\eqref{action}$, which gives the action of $w_j$ on the connection $A_j$ along the flow. This action descends naturally to $S$, so in particular $w^S_j$ solves the equation:
\be
\bar\pl w_j^S=w_jA_0^S-A_j^S w_j^S.\nonumber
\ee
On compact subsets away from $Z_{an}$ the connection terms $A^S_j$ converges in $C^0$ along a subsequence (see \cite{HT}), and we already have uniform convergence of $w^S_j$, thus $\bar\pl w_j$ converges uniformly as well. By working diagonally on an exhaustion of $\ti X:=X\backslash Z_{an}$ by compact sets we can find a subsequence so that $\bar\pl w_j$ converges uniformly everywhere on $\ti X$.  Thus the following limiting equation is satisfied on $\ti X$:
\be
\bar\pl w_\infty-w_\infty A_0^S+A_\infty^S w_\infty=0.\nonumber
\ee
We can concluded that on this set the map $w^S_\infty$ is a holomorphic section of the bundle $Hom(S,S_\infty)$. By Proposition 5.21 from $\cite{Kob}$, because the set $Z_{an}$ has complex codimension at least two, any section $w_\infty^S$ of $Hom(S,S_\infty)$ defined away from $Z_{an}$ must in fact be holomorphic everywhere. This last fact can be viewed as a type of Riemann extension theorem for sections of vector bundles.

As a result the limiting map $w^S_\infty$ is an isomorphism from $S$ to $S_\infty$ and its rank does not drop. Thus the rank of $h^S_\infty$ does not drop. The $C^0$ bound for $\ti h^S_i{}^{-1}$ follows by uniform convergence of $\ti h^S_j$ to $h^S_\infty$.

Next, we show that condition $(C)$ holds, which is the second estimate in the proposition.  Consider the standard representation of a connection in terms of the endomorphisms $h^{S}$:
\be
A^S-A^S_0=(h^S)^{-1}\nabla^0h^S.\nonumber
\ee
The expression on the right is independent of normalization, and so it holds with $h^{S}$ replaced by $\ti h^{S}$.  We have already shown that $\ti h^{S} > cI$, and so it follows immediately that $|A^S-A^S_0|_{H_0}$ is uniformly controlled in time. As a result it is in $L^2$, and an identical $L^2$ bound also follows for $|A^Q-A^Q_0|_{H_0}$. Applying the following inequality from $\cite{UY}$ (which holds since $S$ destabilizes $E$):
\be
|\gamma|^2\leq\int_X{\rm Tr}((\Lambda F-\mu(E)I)|_S)\leq ||\Lambda F||_{L^2(X)}\leq C,\nonumber
\ee
we see that the second fundamental form is universally bounded in $L^2$ along the flow. We now combine these three $L^2$ bounds as follows. Recall the standard decomposition of the connection $A(t)$ onto $S$ and $Q$.  Since each piece from the decomposition is controlled it follows that $|A-A_0|_{H_{0}}$ is bounded in $L^2$ as well.

Working on the bundle $E$, we define the quantity $S=|A-A_0|^2_{H_0}=|\nabla hh^{-1}|^2_{H_0}$. A standard computation along the Donaldson heat flow (see, for example \cite{McF}) yields
\be
(\pl_t-\Delta)S\leq C\,S.\nonumber
\ee
The right hand side is in $L^1$, thus we can apply the parabolic Moser in Theorem~\ref{moser} (see Section~\ref{parabolic} below) to bound $S$ in $C^0$. As a result the second fundamental form is uniformly bounded and hence $(C)$ holds. 
\end{proof}

\begin{proposition}\label{B-AC}
Suppose that $S$ and $Q$ are locally free, and that conditions $(B)$ and $(C)$ hold. Then condition $(A)$ holds as well. 
\end{proposition}
\begin{proof}
We argue on $S$, the proof for $Q$ being identical.   The $C^{0}$ estimate in Proposition~\ref{prop: c0}, combined with an integration by parts on equation~\eqref{eq: key} implies a uniform bound for $\ti h^{S}$ in $L^{2}_{1}(X, H_{0})$.  Combining this with the lower bound from assumption $(B)$, we can apply an argument of Siu \cite{Siu} together with a blow-up argument of Donaldson \cite{Don1} to prove $(A)$.  We refer the reader to \cite{Siu} for the details. The assumption $(C)$ is needed to bound $|\Lambda F^S|$, which is an important part of Siu's argument.  Once condition $(A)$ is verified, Proposition~\ref{prop: lower bd} implies that $(C)$ holds also.
\end{proof}  

Finally, we prove Proposition~\ref{prop: curv est}
\begin{proof}[Proposition~\ref{prop: curv est}]
The $L^{2}_{p}$ estimates of Lemma~\ref{lem: lp2}, together with the lower bound for $\ti h^{S}$ imply that $|F^{S}|_{L^{p}(K, H_{0})}$,  is bounded, and similarly for $|F^{Q}|_{L^{p}(K,H_{0})}$.  Then the upper and lower bounds for $\ti h^{S}, \ti h^{Q}$ show these norms are equivalent to the norms taken with respect to $H_{t}$.  Finally, the estimate for $|\gamma|_{L^{p}(K,H_{t})}$ follows from the upper bound for $|F^{S}|_{L^{p}(K, H_{0})}$ and the inequality
\begin{equation*}
|\gamma|_{H_{t}}^{2} + {\Tr}( \Lambda F^{S}) \leq |\Lambda F(t)|_{H_{t}}\leq C 
\end{equation*}
\end{proof}
In order to achieve an $L^p$ bound on $F$, it remains to control $\nabla \gamma$. We accomplish this by applying parabolic Moser iteration to an evolution inequality we derive in the next section.

\section{Computation of the heat operator}\label{evolve section}

We now compute the heat operator on various important terms. First we note that all norms in this section are taken with respect to the corresponding evolving metric. For us this has the following implication. If $A$ is a fixed endomorphism of $S$, the norm squared is given by
 \be
 |A|^2=\langle A,A\rangle={\rm Tr}(AA^*).\nonumber
 \ee
As stated the adjoint is taken with respect to the evolving metric $H^S$. In a local frame this adjoint can be written explicitly as $(A^*)^\al{}_\b=H^{\al\bar\gamma}\overline{A^\nu{}_\gamma}H_{\bar\nu\b}.$ Thus, taking the time derivative we see $\pl_t (A^*)=[(\Lambda F-\mu(E)I)|_S,A^*].$ We know along the Donaldson heat flow that the term $|\Lambda F|_{H}$ is controlled uniformly from above, which allows us to conclude  $\pl_t |A|^2\leq C|A|^2$. For all other associated bundles in the computations that follow we achieve a similar estimate. Furthermore, when taking the laplacian of a norm, we always use the fact that:
\be
\Delta|A|^2=\langle \Delta A,A\rangle+\langle A,\bar\Delta A\rangle+|\nabla A|^2+|\bar\nabla A|^2.\nonumber
\ee
We need to change $\Delta$ to $\bar \Delta$, and depending on which bundle we are working on this results in curvature terms. Again if $A$ is an endomorphism of $S$ we have $\bar\Delta A=-[\Lambda F^S,A]$. It is important to keep these extra curvature terms in mind, although for us when they show up in an equation they are always absorbed into nearby terms. 

We begin by taking the time derivative of the projection $\pi$:
\be
\dot\pi=\pi(h^{-1}\dot h)(I-\pi)=-\pi\Lambda F(I-\pi)=g^{j\bar k}\nabla_j\nabla_{\bar k}\pi=-g^{j\bar k}\nabla_j\gamma_{\bar k}.\nonumber
\ee
Here the second to last equality follows from the fact that the component of $\Lambda F$ that sends $Q$ to $S$ equals $-g^{j\bar k}\nabla_j\nabla_{\bar k}\pi$, and the last equality follows since the second fundamental form is given by $\gamma_{\bar k}=-\nabla_{\bar k}\pi$. This allows us to compute the time derivative of the second fundamental form $\gamma$:
\bea
\pl_t(\gamma_{\bar k})&=&-\nabla_{\bar k}\dot\pi\nonumber\\
&=&g^{\ell\bar m}\nabla_{\bar k}\nabla_\ell\gamma_{\bar m}\nonumber\\
&=&g^{\ell\bar m}([\nabla_{\bar k},\nabla_\ell]\gamma_{\bar m}+\nabla_\ell\nabla_{\bar k}\gamma_{\bar m}).\nonumber
\eea
Note $\nabla_{\bar k}\gamma_{\bar m}=\nabla_{\bar m}\gamma_{\bar k}$, since $\gamma$ is $\bar\pl$ closed. Because $\gamma$ is a morphism from $Q$ to $S$ taking the commutator of derivatives gives:
\bea
\label{gammaevolve}
\pl_t(\gamma_{\bar k})&=&g^{\ell\bar m}(R_{\bar k\ell}{}^{\bar p}{}_{\bar m}\gamma_{\bar p}-F^S_{\bar k\ell}\gamma_{\bar m}+\gamma_{\bar m}F^Q_{\bar k\ell})+\Delta\gamma_{\bar k}.
\eea
Thus we can take the derivative of the norm squared of $\gamma$:
\bea
\pl_t|\gamma|^2&\leq&\langle \dot\gamma,\gamma\rangle+\langle \gamma,\dot\gamma\rangle+C|\gamma|^2\nonumber\\
&\leq&C|\gamma|^2(1+|F^S|+|F^Q|)+\langle \Delta \gamma,\gamma\rangle+\langle \gamma,\Delta \gamma\rangle.\nonumber
\eea
In conclusion the heat operator on $|\gamma|^2$ is bounded  by:
\be
\label{gamma}
(\pl_t-\Delta)|\gamma|^2\leq C|\gamma|^2(1+|F^S|+|F^Q|).
\ee
Next we compute the heat operator on $|\nabla\gamma|^2$ starting with the time derivative of $\nabla_j\gamma_{\bar k}$. Suppose $\nabla^S$ is the covariant derivative on $S$. Then the time derivative of the connection is given by
\be
\dot \nabla_j^S=\nabla_j^S(h^{-1}h|_{S})=\nabla_j^S(\Lambda(F^S+\gamma\wedge\gamma^\dagger)).\nonumber
\ee
Now, $\gamma_{\bar k}$ is a morphism from $Q$ to $S$, thus to differentiate it we need the covariant derivative on $Hom(Q,S)$. Taking the time derivative we see:
\be
\pl_t(\nabla_j^{Hom(Q,S)}\gamma_{\bar k})=\nabla_j^S(\Lambda F^S+\Lambda\gamma\wedge\gamma^\dagger)\gamma_{\bar  k}+\gamma_{\bar k}\nabla_j^Q(\Lambda F^Q-\Lambda\gamma\wedge\gamma^\dagger)+\nabla_j\dot\gamma_{\bar k}.\nonumber
\ee
Combing this with \eqref{gammaevolve} yields:
\be
\nabla_j\dot\gamma_{\bar k}=g^{\ell\bar m}\nabla_j\left(R_{\bar k\ell}{}^{\bar p}{}_{\bar m}\gamma_{\bar p}-F^S_{\bar k\ell}\gamma_{\bar m}+\gamma_{\bar m}F^Q_{\bar k\ell}\right)+\nabla_j\Delta\gamma_{\bar k}.\nonumber
\ee
Now, we want to interchange the order of $\Delta$ and $\nabla_j$:
\be
g^{\ell\bar p}\nabla_j\nabla_{\ell}\nabla_{\bar p}\gamma_{\bar m}=g^{\ell\bar p}\nabla_\ell\nabla_{j}\nabla_{\bar p}\gamma_{\bar m}=\Delta\nabla_{j}\gamma_{\bar m}-g^{\ell\bar p}\nabla_{\ell}(R_{\bar p j}{}^{\bar q}{}_{\bar m}\gamma_{\bar q}+F_{\bar p j}^S\gamma_{\bar m}-\gamma_{\bar m}F^Q_{\bar pj}).\nonumber
\ee
Putting everything together we bound the time derivative in the following estimate:
\bea
\pl_t|\nabla\gamma|^2&\leq&\langle \pl_t(\nabla\gamma),\nabla\gamma\rangle+\langle \nabla\gamma,\pl_t(\nabla\gamma)\rangle+C|\nabla\gamma|^2\nonumber\\
&\leq&C|\nabla\gamma|^2(1+|\gamma|^2+|F^S|+|F^Q|)\nonumber\\
&&+|\gamma||\nabla\gamma|(|\nabla F^S|+|\nabla F^Q|)+\langle \Delta \gamma,\gamma\rangle+\langle \gamma,\Delta \gamma\rangle\nonumber
\eea
Thus the heat operator is controlled by:
\bea
\label{nablagamma}
(\pl_t-\Delta)|\nabla\gamma|^2&\leq&C|\nabla\gamma|^2(1+|\gamma|^2+|F^S|+|F^Q|)\nonumber\\
&&+|\gamma||\nabla\gamma|(|\nabla F^S|+|\nabla F^Q|)-|\nabla\nabla \gamma|^2-|\bar\nabla\nabla\gamma|^2.
\eea
Next we turn to the curvature term $|F^S|^2$. Note that for any path of metrics $h^S(t)$, the time derivative of the curvature is given by $\pl_t F^S_{\bar kj}=-\nabla_{\bar k}\nabla_j((h^S)^{-1}\pl_t h^S)$ (for details see $\cite{Siu}$). Thus along the Donaldson heat flow we have $\pl_t F^S_{\bar kj}=\nabla_{\bar k}\nabla_j(\Lambda F|_S)$. Applying the derivative to the norm squared yields:
\bea
\pl_t|F^S|^2&\leq&2g^{j\bar k}g^{\ell\bar m}{\rm Tr}(\pl_tF_{\bar mj}^S(F_{\bar \ell k}^S)^*)+C|F^S|^2\nonumber\\
&=&2g^{j\bar k}g^{\ell\bar m}{\rm Tr}(\nabla_{\bar m}\nabla_j(\Lambda F|_S)(F_{\bar \ell k}^S)^*)+C|F^S|^2\nonumber\\
&=&2g^{j\bar k}g^{\ell\bar m}{\rm Tr}(\nabla_{\bar m}\nabla_j(\Lambda F^S+\Lambda \gamma\wedge\gamma^\dagger)(F_{\bar \ell k}^S)^*)+C|F^S|^2\nonumber\\
&\leq&2g^{j\bar k}g^{\ell\bar m}{\rm Tr}(\nabla_{\bar m}\nabla_j(\Lambda F^S)(F_{\bar \ell k}^S)^*)+(|\bar\nabla\nabla\gamma||\gamma|+|\nabla\gamma|^2)|F^S|+C|F^S|^2.\nonumber
\eea
Applying the second Bianchi identity to the first term on the right we can get a laplacian out of it, but at the cost of an extra curvature term:
\bea
\nabla_{\bar m}\nabla_j(\Lambda F^S)&=&g^{p\bar q}\nabla_{\bar m}\nabla_jF^S_{\bar qp}\nonumber\\
&=&g^{p\bar q}\nabla_{\bar m}\nabla_pF^S_{\bar qj}\nonumber\\
&=&g^{p\bar q}[\nabla_{\bar m},\nabla_p]F^S_{\bar qj}+g^{p\bar q}\nabla_p\nabla_{\bar q}F^S_{\bar mj}.\nonumber
\eea
It follows that:
\be
g^{j\bar k}g^{\ell\bar m}{\rm Tr}(\nabla_{\bar m}\nabla_j(\Lambda F^S)(F_{\bar k\ell}^S)^*)\leq 2g^{j\bar k}g^{\ell\bar m}{\rm Tr}(\Delta(F^S_{\bar m j})(F_{\bar k\ell}^S)^*)+C|F^S|^3.\nonumber
\ee
Thus
\be
\label{FS}
(\pl_t-\Delta)|F^S|^2\leq C(|F^S|^2+|F^S|^3)+(|\bar\nabla\nabla\gamma||\gamma|+|\nabla\gamma|^2)|F^S|-|\nabla F^S|^2.
\ee

The computation of the heat operator applied to $|F^Q|^2$ follows in exactly the same fashion as for $|F^S|^2$ above. Thus we conclude:
\be
\label{FQ}
(\pl_t-\Delta)|F^Q|^2\leq C(|F^Q|^2+|F^Q|^3)+(|\bar\nabla\nabla\gamma||\gamma|+|\nabla\gamma|^2)|F^Q|-|\nabla F^Q|^2.
\ee

\section{The Parabolic Moser Iteration}
\label{parabolic}

In this section, we prove a version of Moser's $C^{0}$ estimate for parabolic equations.  When $p\geq 2$, Theorem~\ref{moser} is essentially the same as Theorem 2.1 in \cite{Ye}, with some minor modifications for our setting.  For $p\in(0,2)$ we adapt the standard argument from the elliptic case; see for instance \cite{LH}.  We include the proof for the reader's convenience.
\begin{theorem}\label{moser}
Fix a point $x_{0} \in X\backslash Z_{alg}$, and fix a compact set $K\subset X\backslash Z_{alg}$ so that $x\in K$.  Let $2R=\text{dist}(x_{0},\del K)$.  Suppose that $u$ is a nonnegative Lipschitz function on $X \times [0,T]$ satisfying
\begin{equation*}
(\ddt  - \Delta) u \leq \Theta u
\end{equation*}
for a non-negative function $\Theta (x,t) \in L^{q}(K)$, for some $q \gg n$.  Suppose moreover that there is a positive constant $0<A<\infty$ so that
\begin{equation*}
\sup_{[0,T]}\|\Theta \|_{L^{q}(K)}(t) \leq A.
\end{equation*}
Then, on $B(x_{0}, R) \times [0,T]$, and for any $p >0$, we have the estimate
\begin{equation*}
\begin{aligned}
|u(x,t)| \leq C(g,n,q,p)&\left[A^{n/(q-n)} + \frac{1}{t} + \frac{1}{R^{2}}\right]^{\frac{n+1}{p}} \\
& \cdot (\int_{0}^{T} \int_{B(x_{0},2R)} u^{p} dV dt)^{\frac{1}{p}}.
\end{aligned}
\end{equation*}
\end{theorem}
\begin{proof}

We prove the theorem first for $p\geq2$.  The estimate for $p\in(0,2)$ is obtained by a standard scaling argument.  We begin by computing
\begin{equation*}
\frac{1}{p}\ddt \int u^{p}\eta^{2} dV \leq \int (\Theta u^{p})\eta^{2} - \nabla(\eta^{2}u^{p-1}) \cdot \nabla u dV
\end{equation*}
We show how to deal with the first term.  Using that $\Theta $ is uniformly bounded in $L^{q}(B(x_{0},2R)) = L^{q}(B)$ for large $q >n$, we have
\begin{equation*}
\begin{aligned}
\int \Theta  u^{p} \eta^{2} &= \int \Theta  (\eta u^{p/2})^{2} \leq \left(\int_{B(x_{0},2R)} \Theta ^{q}\right)^{1/q} \left( \int (\eta u^{p/2})^\frac{2q}{q-1} \right) ^{1-\frac{1}{q}}\\
&\leq \| \Theta \|_{L^{q}(B)} \left( \int (\eta u^{p/2})^\frac{2q}{q-1} \right) ^{1-\frac{1}{q}}
\end{aligned}
\end{equation*}
Now, since $q>n$, we have $2^{*} = \frac{2n}{n-1} > \frac{2q}{q-1} > 2$.  Hence, using the interpolation inequality, and the Sobolev inequality we have
\begin{equation*}
\| \eta u^{p/2} \|_{L^{\frac{2q}{q-1}}(X)} \leq \epsilon \| \eta u^{p/2} \|_{L^{2*}(X)} + C(n,q)\epsilon^{-\frac{n}{n-q}}\|\eta u^{p/2}\|_{L^{2}}
\end{equation*}
for any $\epsilon >0$.  Choose $\epsilon = \sqrt{\frac{1}{2p \|\Theta \|_{L^{q}(B)}}}$.  Then we obtain  
\begin{equation*}
\int \Theta  u^{p} \eta^{2} \leq \frac{1}{p} \int |\nabla(\eta u^{p/2})|^{2} + C(n,q)\|\Theta \|^{n/(q-n)}_{L^{q}(B)}p^{n/ (q-n)} \|\eta u^{p/2}\|^{2}_{L^{2}}.
\end{equation*}
The remaining terms we treat as follows.
\begin{equation*}
\begin{aligned}
-\int  \nabla(\eta^{2}u^{p-1}) \cdot \nabla u dV &= \frac{-4(p-1)}{p^{2}} \int \eta^{2} |\nabla u^{p/2}|^{2}dV - \int 2u^{p-1}\eta \nabla \eta \cdot \nabla u dV \\
&= \frac{-4(p-1)}{p^{2}} \left[ \int |\nabla( \eta u^{p/2})|^{2} + u^{p}|\nabla \eta|^{2} - 2 u^{p/2}\nabla(\eta u^{p/2})\cdot \nabla \eta dV \right] \\
&\qquad -  \int 2u^{p-1}\eta \nabla \eta \cdot \nabla u dV
\end{aligned}
\end{equation*}
The final term we write as
\begin{equation*}
\begin{aligned}
2u^{p-1}\eta\nabla \eta \cdot \nabla u = \frac{4}{p} \left[ u^{p/2}\nabla(u^{p/2} \eta) \cdot \nabla \eta - u^{p}|\nabla \eta|^{2}\right]
\end{aligned}
\end{equation*}
In summation, we obtain
\begin{equation*}
\begin{aligned}
-\int  \nabla(\eta^{2}u^{p-1}) \cdot \nabla u dV &= \frac{-4(p-1)}{p^{2}} \int |\nabla( \eta u^{p/2})|^{2} dV\\
&+ \frac{4}{p^{2}}\int u^{p}|\nabla \eta|^{2} dV + \frac{4(p-2)}{p^{2}} \int u^{p/2}\nabla(u^{p/2} \eta) \cdot \nabla \eta dV\\
&\leq -\frac{2}{p} \int  |\nabla( \eta u^{p/2})|^{2} dV + \frac{2}{p} \int u^{p}|\nabla \eta|^{2} dV
\end{aligned}
\end{equation*}
Putting all of this together, we get

\begin{equation}\label{iterate gen}
\begin{aligned}
\ddt \int u^{p}\eta^{2}dV + \int  |\nabla( \eta u^{p/2})|^{2} dV &\leq  2\int u^{p}|\nabla \eta|^{2} dV\\ & \qquad + C(n,q)\|\Theta \|^{n/(q-n)}_{L^{q}(B)}p^{q/ (q-n)} \int u^{p}\eta^{2} dV
\end{aligned}
\end{equation}
Fix $T>2$, and choose $0 < \tau < \tau' < T$, and define
\begin{equation*}
\psi(t) =  \left \{ 
	\begin{array}{ll}
	0 &: 0\leq t \leq \tau\\
	(t-\tau)/(\tau'-\tau) &: \tau \leq t \leq \tau' \\
	1 &: \tau' \leq t \leq T
	\end{array}
	\right.
\end{equation*}
Multiplying equation~\eqref{iterate gen} by $\psi$, we compute
\begin{equation*}
\begin{aligned}
\ddt \left( \psi \int u^{p}\eta^{2} \right) +\psi \int  |\nabla( \eta u^{p/2})|^{2} dV &\leq 2\psi\int u^{p}|\nabla \eta|^{2} dV \\
& + \left(C(n,q)\|\Theta \|^{n/(q-n)}_{L^{q}(B)}p^{q/ (q-n)}\psi + \psi'\right) \int u^{p}\eta^{2} dV
\end{aligned}
\end{equation*}
To simplify notation, let us institute $A =C(n,q)\|\Theta \|^{n/(q-n)}_{L^{q}(B)}$. Fix any $s \geq \tau'$, and integrate from $0$ to $s$ to obtain
\begin{equation}\label{iterate gen 1}
\begin{aligned}
\int u^{p}(s)\eta^{2} dV & + \int_{\tau'}^{s}\int  |\nabla( \eta u^{p/2})|^{2} dV dt \\
&\leq 2 \int_{\tau}^{T} \int u^{p}|\nabla \eta|^{2} dV dt + \left(Ap^{q/ (q-n)} + \frac{1}{\tau'-\tau}\right) \int_{\tau}^{T} \int u^{p}\eta^{2}.
\end{aligned}
\end{equation}
We now apply H\"older's inequality and the Sobolev inequality to see that
\begin{equation}
\begin{aligned}
\int_{\tau'}^{T} \int u^{p(1+\frac{1}{n})}\eta^{2+\frac{2}{n}} &\leq \int_{\tau'}^{T} \left( \int u^{p}\eta^{2} dV\right)^{1/n} \left(\int u^{\frac{pn}{n-1}}\eta^{\frac{2n}{n-1}}dV \right)^{\frac{n-1}{n}} dt\\
& \leq C(g) \left( \sup_{\tau' \leq t \leq T} \int u^{p}\eta^{2}\right)^{1/n}\int_{\tau'}^{T} \int |\nabla(\eta u^{p/2})|^{2} dV dt\\
&\leq C(g) \left[ 2 \int_{\tau}^{T} \int |\nabla \eta|^{2} u^{p} dV dt  +  \left(Ap^{q/ (q-n)}  + \frac{1}{\tau'-\tau}\right)\int_{\tau}^{T} \int u^{p}\eta^{2} dV dt \right] ^{1+ \frac{1}{n}}
\end{aligned}
\end{equation}
where the last line follows from the estimate~\eqref{iterate gen 1}.  Set 
\begin{equation*}
H(p,\tau,R) = \int_{\tau}^{T} \int_{B(x_{0}, R)} u^{p} dV dt
\end{equation*}
for $0 < \tau <T$ and $0 < R < \text{dist}(x_{0}, Z_{alg})$.  Given $0<R' <R < \text{dist}(x_{0}, Z_{alg})$ we define 
\begin{equation*}
\eta(x) =  \left \{ 
	\begin{array}{ll}
	1 &: x \in B(x_{0}, R')\\
	1 - \frac{1}{R-R'}(d(x_{0},x) -R) &: x \in B(x_{0}, R) \backslash B(x_{0}, R')\\
	0 &: x \in B(x_{0}, R)^{c}
	\end{array}
	\right.
\end{equation*}
Observe that $|\nabla \eta | \leq \frac{1}{R-R'}$.  Then we have proved
\begin{equation*}
H(p(1+\frac{1}{n}), \tau', R') \leq C(g)\left[ Ap^{q/q-n} + \frac{1}{\tau'-\tau} + \frac{1}{(R-R')^{2}}\right]^{1+\frac{1}{n}} H(p,\tau, R)^{1+\frac{1}{n}}
\end{equation*}
Now, fix $R = \frac{1}{2} \text{dist}(x_{0}, Z_{alg})$, choose $t \in(0,T]$, and set $\mu = 1+ \frac{1}{n}$, $p_{k} = \mu^{k}p_{0}$, $\tau_{k} = (1-\frac{1}{\mu^{k+1}})t$.  Choose $\delta \in (0,1)$ and set $R_{k} = R(\delta+(1-\delta)\mu^{-k})$.   Then we have
\begin{equation*}
\begin{aligned}
H(&p_{k+1}, \tau_{k+1}, R_{k+1})^{\frac{1}{p_{k+1}}} \\
&\leq C(g)^{\frac{1}{p_{k+1}}}\left[Ap^{q/(q-n)}\mu^{kq/(q-n)} + \mu^{k}\frac{\mu}{t(\mu-1)} +\mu^{2k} \frac{4\mu^{2}}{R^{2}(1-\delta)^{2}(\mu-1)^{2}} \right]^{\frac{1}{p_{k}}} H(p_{k}, \tau_{k}, R_{k})^{\frac{1}{p_{k}}}.
\end{aligned}
\end{equation*}
For simplicity, we assume that $q>2n$ so that $q/(q-n) <2$.  Then we have
\begin{equation*}
\begin{aligned}
H(&p_{k+1}, \tau_{k+1}, R_{k+1})^{\frac{1}{p_{k+1}}} \\
&\leq C(g)^{\frac{1}{p_{k+1}}}\left[Ap^{q/(q-n)} +\frac{\mu}{t(\mu-1)} +\frac{4\mu^{2}}{R^{2}(1-\delta)^{2}(\mu-1)^{2}} \right]^{\frac{1}{p_{k}}}\mu^{\frac{2k}{p_{k}}} H(p_{k}, \tau_{k}, R_{k})^{\frac{1}{p_{k}}}.
\end{aligned}
\end{equation*}
Iterating, we obtain that, for all $(x,t) \in B(x_{0}, \delta R) \times [0,T]$, we have
\begin{equation*}
\begin{aligned}
|u(x,t)| \leq C(g,n)^{\frac{1}{p}} &\left[C(n,q)\|\Theta \|^{n/(q-n)}_{L^{q}(B(x_{0},R)}p^{q/(q-n)} + \frac{c(n)}{t} + \frac{c(n)}{(1-\delta)^{2}R^{2}}\right]^{\frac{n+1}{p}} \\
& \cdot (\int_{0}^{T} \int_{B(x_{0},2R)} u^{p} dV dt)^{\frac{1}{p}}
\end{aligned}
\end{equation*}
In order to obtain the estimate for $p\in (0,2)$, we use a scaling argument, which is adapted from the argument for elliptic equations; see \cite{LH}.
Fix some $T$, and let $Q(r,t) = B(x_{0},r) \times [t, T]$, and from now on suppress the dependence on $x_{0}$.  The estimate we have just proven shows that if $R > R'$ and $t<t'<T$ then
\begin{equation*}
\sup_{Q(R',t')}u \leq C\left[ \frac{1}{(t'-t)} + \frac{1}{(R-R')^{2}}+B\right]^{\frac{n+1}{2}} \left(\int_{Q(R,t)} u^{2}\right)^{1/2}
\end{equation*}
For $p\in (0,2)$, we write
\begin{equation*}
 \left(\int_{Q(R,t)} u^{2}\right)^{1/2} \leq \sup_{Q(R,t)}u^{1-\frac{p}{2}}\left(\int_{Q(R,t)}u^{p}\right)^{1/2}.
\end{equation*}
By H\"older's inequality, we obtain
\begin{equation*}
\sup_{Q(R',t')}u \leq \frac{1}{2}\sup_{Q(R,t)}u + C\left[ \frac{1}{(t'-t)} + \frac{1}{(R-R')^{2}}+B\right]^{\frac{n+1}{p}} \left(\int_{Q(R,t)} u^{p}\right)^{1/p}
\end{equation*}
Let $f(X,y) := \sup_{Q(X,y)}u$.  Then the above inequality says that if $X'<X$ and $T>y'>y$, then
\begin{equation*}
f(X',y') \leq \frac{1}{2} f(X,y) + A\left[ \frac{1}{(y'-y)} + \frac{1}{(X-X')^{2}}+B\right]^{\frac{n+1}{p}}.
\end{equation*}
We are reduced to proving the following lemma, which is the analog of Lemma 4.3 in \cite{LH}.
\begin{lemma}
Let $f(X,y) \geq 0$ be bounded in $ \Sigma:= [L_{0},L_{1}] \times [s_{1},s_{0}]\subset R^{2}_{\geq 0}$.  Suppose that, for $(X,y), (X',y')  \in \Sigma$ satisfying $X'<X$ and $y'>y$ we have
\begin{equation*}
f(X',y') \leq \theta f(X,y) + A\left[ \frac{1}{(y'-y)} + \frac{1}{(X-X')^{2}}+B\right]^{\alpha}
\end{equation*}
for some $\theta \in [0,1)$.  Then for any $(X,y), (X',y')  \in \Sigma$ satisfying $X'<X$ and $y'>y$ there holds
\begin{equation*}
f(X',y') \leq c(\alpha, \theta)A \left[ \frac{1}{(y'-y)} + \frac{1}{(X-X')^{2}} +B\right]^{\alpha}.
\end{equation*}
\end{lemma}
\begin{proof}
First, choose a number $\tau$ satisfying $\theta <\tau^{\alpha} <1$.  Define a sequence by $R_{0} = X'$ and $t_{0} = y'$, and set
\begin{equation*}
R_{i+1} := R_{i} + (1-\tau)\tau^{i}(X-R_{0}) \quad t_{i+1} = t_{i}-(1-\tau^{2})\tau^{2i}(t_{0}-y).
\end{equation*}
Note that $R_{\infty} = X$ amd $t_{\infty} =y$.  Note that $\tau <1$ so that $(1-\tau)^{2} < (1-\tau^{2})$.  Then the assumption implies
\begin{equation*}
f(R_{i},t_{i}) \leq \theta f(R_{i+1},t_{i+1}) + \frac{A\tau^{-2i\alpha}}{(1-\tau)^{2\alpha}}\left[\frac{1}{(t_{0}-y)} + \frac{1}{(X-R_{0})^{2}} + B\right]^{\alpha}.  
\end{equation*}
Iterating this we obtain
\begin{equation*}
f(R_{0},t_{0}) \leq \theta^{k}f(R_{k},t_{k}) + \frac{A}{(1-\tau)^{2\alpha}}\left[\frac{1}{(t_{0}-y)} + \frac{1}{(X-R_{0})^{2}} + B\right]^{\alpha}\cdot \left(\sum_{i=0}^{k-1}\theta^{i}\tau^{-2i\alpha}\right)
\end{equation*}
By our choice of $\tau$, we can take the limit as $k\rightarrow \infty$ to prove the lemma.
\end{proof}
The proof of the theorem follows immediately.
\end{proof}

We can now prove the main theorem.
\begin{proof}[Proof of Theorem~\ref{ZanZalg}]
We begin by defining
\begin{equation*}
u := |\nabla \gamma|^{2}_{H_{t}} + |F^{S}|_{H_{t}}^{2} + |F^{Q}|_{H_{t}}^{2}+1
\end{equation*}
Since $\nabla \gamma$ is a component of the curvature, it follows that $|\nabla \gamma|^{2}_{H_{t}}$ is uniformly bounded in $L^{1}(X)$.  By the estimates in Proposition~\ref{prop: curv est}, on any compact set $K \subset Z_{alg}$ the function $u$ is uniformly bounded in $L^{1}(K)$, independent of time.  Using equations~\eqref{FS}~\eqref{FQ}~\eqref{nablagamma},  we compute
\begin{equation*}
\begin{aligned}
(\del_{t}-\Delta) u &\leq C_{2} |\nabla \gamma|^{2}(1+|\gamma|^{2} + |F^{S}|+|F^{Q}|)\\ &+ |\nabla \gamma|\left(|\gamma|^{2}+ |\gamma||\nabla F^{S}| + |\gamma||\nabla F^{Q}|\right) - |\nabla \nabla \gamma|^{2} - |\bar{\nabla}\nabla \gamma|^{2}\\
&+C_{3}(|F^{S}|^{2}+ |F^{S}|^{3}) + \left(|\nabla\nabla \gamma||\gamma| + |\nabla \gamma|^{2}\right)|F^{S}| - |\nabla F^{S}|^{2}\\
&+C_{4}(|F^{Q}|^{2}+ |F^{Q}|^{3}) + \left(|\nabla\nabla \gamma||\gamma| + |\nabla \gamma|^{2}\right)|F^{Q}| - |\nabla F^{Q}|^{2}\\
&\leq C_{2} |\nabla \gamma|^{2}(1+|\gamma|^{2} + |F^{S}|+|F^{Q}|)\\ &+ |\nabla \gamma||\gamma|^{2}+2|\nabla \gamma|^{2} |\gamma|^{2} + \frac{1}{2}\left(|\nabla F^{S}|^{2} + |\nabla F^{Q}|^{2}\right) - |\nabla \nabla \gamma|^{2} - |\bar{\nabla}\nabla \gamma|^{2}\\
&+C_{3}(|F^{S}|^{2}+ |F^{S}|^{3}) + \frac{1}{2}|\nabla\nabla \gamma|^{2} + 2|\gamma|^{2}|F^{S}|^{2} + |\nabla \gamma|^{2}|F^{S}| - |\nabla F^{S}|^{2}\\
&+ C_{4}(|F^{Q}|^{2}+ |F^{Q}|^{3}) + \frac{1}{2}|\nabla\nabla \gamma|^{2} + 2|\gamma|^{2}|F^{Q}|^{2} + |\nabla \gamma|^{2}|F^{Q}| - |\nabla F^{Q}|^{2}\\
&\leq P(|\gamma|, |F^{S}|, |F^{Q}|) u
\end{aligned}
\end{equation*}
where $P(x,y,z)$ denotes some algebraic function of $x,y,z$.  Fix any compact set $K\subset X \backslash Z_{alg}$.  Since $|F^{S}|, |F^{Q}|, |\gamma|$ are uniformly bounded in $L^{p}(K)$ for all $p\gg 1$, it follows that
\begin{equation*}
\sup_{t\in [0,\infty)}\|P(|\gamma|, |F^{S}|, |F^{Q}|)\|_{L^{100n}(K)}(t) \leq C(n,K).
\end{equation*}
We can now apply the Moser iteration.  Fix $x_{0} \in X\backslash Z_{alg}$ and set $2R=dist(x_{0},Z_{alg})$.  Suppose that there exists times $t_{j} \rightarrow \infty$ such that  $\sup_{B(x_{0}, \frac{R}{2})}u(x,t_{j}) \geq j$.  Set $u_{j} = u(t-(t_{j}+1))$.  Then the estimate in Theorem~\ref{moser} implies
\begin{equation*}
j\leq \sup_{B(x_{0}, \frac{R}{2})\times[1,2]} u_{j}\leq C(n,R)\left[1+\frac{1}{R^{2}}\right]^{n+1}\int_{0}^{2}\int_{B(x_{0},R)}u_{j} dVdt
\end{equation*}
Since $B(x_{0}, R) \Subset X\backslash Z_{alg}$, the $L^{1}$ norm on the right hand side is uniformly bounded.  Thus we have a contradiction for $j$ sufficiently large.
\end{proof}

\end{normalsize}

\bigskip
\centerline{}

Tristan C. Collins

Department of Mathematics, Columbia University, New York, NY 10027

\bigskip

Adam Jacob

Department of Mathematics, Harvard University, Cambridge, MA 02138

\end{document}